\documentclass[12pt]{amsart}
\usepackage{amstext, xspace, amssymb} 
\def\q{\qquad}

\def\SP{\vskip 5pt}

\newcommand\ds{\displaystyle}

\newtheorem{thm}{Theorem}[section]
\newtheorem{prop}[thm]{Proposition}
\newtheorem{lem}[thm]{Lemma}
\newtheorem{rem}[thm]{Remark}

\newtheorem{defn}[thm]{Definition}

\def\Apio{\mathcal{A}_p(\Omega)}
\def\ApDio{\mathcal{A}_{p,D}(\Omega)}
\def\mb{\it}    

\def\mbC{{\mathbb C}}

\def\mbf{{\it f}}      
\def\mbg{{\it g}}
\def\mbR{{\mathbb R}}
\def\mbbR{{\mathbb R}}
\def\mbbO{\mathbb{H}^{1,p}(\Omega)}
\def\mbbH{\mathbb{H}_{0}^{1,p}(\mbR^3)}
\def\mbbHn{\mathbb{H}^{1,p}(\mbR^3)}
\def\mbbHO{\mathbb{H}_{0}^{1,p}(\Omega)}

\def\mbg{{\mb g}}

\def\mbP{{\it \phi}}      
\newcommand\Con{C_0^{\infty}({\mathbb R}^3)}
\def\Cono{C_0^\infty(\Omega)}

\def\cF{{\mathcal F}}
\def\coF{\overline{\mathcal F}}
\def\cS{{\mathcal S}}
\def\ep{\epsilon}
\def\<{\langle}
\def\>{\rangle}
\def\pp{\hbox{\sl p}}

\def\pa{\partial}

\def\WpO{\mathbb{W}^{1,p}(\Omega)}

\def\<{\langle}
\def\>{\rangle}

\makeatletter
   
          \@addtoreset{equation}{section}
\makeatother

\numberwithin{equation}{section}

\begin{document}

\baselineskip=1.3\baselineskip

\title[Dirac--Sobolev spaces]{Dirac--Sobolev spaces and Sobolev spaces}

\author{Takashi Ichinose and Yoshimi Sait\={o} \\}

\address{Department of Mathematics, Kanazawa University, Kanazawa, 920-1192, Japan}
\email{ichinose@kenroku.kanazawa-u.ac.jp}
\address{Department of Mathematics,
     University of Alabama at Birmingham, Birmingham, AL 35294, USA}
\email{saito@math.uab.edu}



\begin{abstract}
        The aim of this work is to study the first order Dirac-Sobolev spaces
     in $L^p$ norm on an open subset of ${\mathbb R}^3$ to clarify its relationship
     with the corresponding Sobolev spaces. It is shown that for $1< p <\infty$,
     they coincide, while for $p=1$, the latter spaces are proper subspaces of the former.
\end{abstract}

\maketitle



\vskip 10pt

\section{\bf Introduction}                                

\SP

        In the recent work \cite{BES},  Balinski-Evans-Sait\=o introduced
     an $L^p$-seminorm $\|(\alpha\cdot \pp){\mb f}\|_{p,\Omega}$ of
     a ${\mbC}^4$-valued function $f$ in an open subset $\Omega$ of
     ${\mbR}^3$, relevant to  a massless  Dirac operator
\begin{equation}\label{1dirac}                                  
       \alpha \cdot \pp =
          \sum_{j=1}^3 \alpha_j(-i\partial_j) \q (\partial_j
                                              = \partial/\partial{x_j}),
\end{equation}
     where $\pp = -i\nabla$, and
     $\alpha= (\alpha_1, \, \alpha_2, \, \alpha_3)$ is the triple of
     $4 \times 4$ Dirac matrices
\begin{equation}\label{1alpha}                                    
    \alpha_j =
  \begin{pmatrix}
     0_2  &\sigma_j \\ \sigma_j &   0_2
  \end{pmatrix}  \qquad  (j = 1, \, 2, \, 3)
\end{equation}
     with  the $2\times 2$ zero matrix $0_2$ and the triple of  $2 \times 2$
     Pauli matrices
$$
    \sigma_1 =
\begin{pmatrix}
       0&1 \\ 1& 0
\end{pmatrix}, \,\,\,
    \sigma_2 =
\begin{pmatrix}
       0& -i  \\ i&0
\end{pmatrix}, \,\,\,
    \sigma_3 =
\begin{pmatrix}
       1&0 \\ 0&-1
\end{pmatrix}.
$$
     They used this seminorm to give a group
     of inequalities called {\it Dirac--Sobolev inequalities}
     in order to obtain $L^p$-estimates of the {\it zero modes}, i.e. eigenfunctions
     associated with the eigenvalue $\lambda = 0$,
     of the Dirac operator $(\alpha\cdot {\pp})+Q$, where $Q(x)$ is a $4\times 4$
     Hermitian matrix-valued potential decaying at infinity.
     We believe that our above notation ``$\pp\,$" for the differential
     operator $-i\nabla$ will not be confused with another ``$p$" which appears
     as the superscript $1\leq p <\infty$ of the space $L^p$.

\par
        Let
     $\Omega$ be an open subset of ${\mbR}^3$ and let the first order Dirac--Sobolev space
     $ \mbbHO$, $1\le p <\infty,$ be the
     completion of $[C_0^{\infty}(\Omega)]^4$ with respect to the norm
\begin{equation}\label{1H1pOmega}                                   
    \|{\mb f}\|_{D,1,p,\Omega} :
        = \left \{ \int_{\Omega} (|{\mb f}(x)|_p^p +
          |(\alpha \cdot \pp){\mb f}(x)|_p^p \,)\,  dx \right \}^{1/p}
        = \big(\|{\mb f}\|_{p,\Omega}^p
               +\|(\alpha \cdot \pp){\mb f}\|_{p,\Omega}^p\big)^{1/p} ,
\end{equation}
 where ${\mbf}(x) = {}^t(f_1(x), f_2(x), f_3(x), f_4(x))$, the norm of a vector
${\mb a} = {}^t(a_1, a_2, a_3, a_4) \in {\mathbb C}^4$ being denoted by
\begin{equation} \label{1gpnorm}                                      
      | {\mb a} |_p = \Big[\sum_{k=1}^4 |a_k|^p\, \Big]^{1/p}. \hskip 70pt
\end{equation}

        As one of the simplest Dirac--Sobolev inequalities (\cite{BES}, Corollary 2), they
     showed: If $\Omega$ is a bounded open
     subset of ${\mbbR}^3$ and ${\mb f} \in \mathbb{H}_{0}^{1,p}(\Omega)$
     with $1 \le p < \infty$, then for $1 \le k < p(p+3)/3$
     there exists a positive constant $C$ such that
\begin{equation}\label{1DS}                                         
         \|{\mb f}\|_{k,\Omega}
                  \le C \|(\alpha \cdot \pp){\mb f}\|_{p, \Omega}\,,
\end{equation}
     where  $\| {\mb g} \|_{p,\Omega}$ stands for the norm of
     ${\mb g} = {}^t(g_1, g_2, g_3, g_4) \in [L^p(\Omega)]^4$ given by
\begin{equation}\label{1pOnorm}                                     
     \|{\mb g}\|_{p,\Omega} =
          \left \{ \int_{\Omega} |{\mb g}(x)|_p^p\, dx \right \}^{1/p}
           = \left\{ \int_{\Omega}\, \sum_{j=1}^4 |g_j(x)|^p \,dx \right\}^{1/p}.
\end{equation}

       Now let $\beta$ be the fourth Dirac matrix given by
\begin{equation}\label{1betam}                                         
       \beta =
  \begin{pmatrix}
     {1}_2 & 0_2 \\ 0_2 &   -{1}_2
  \end{pmatrix},
\end{equation}
     where $1_2$ is the $2 \times 2$ unit matrix. It has been known that
     the free massless Dirac operator $\alpha\cdot \pp$ or the free Dirac
     operator $\alpha\cdot \pp + m\beta$ with positive mass $m$ and
     the relativistic Schr\"odinger operator $\sqrt{m^2 - \Delta}$ may bring similar
     properties in $L^2$ sometimes but not necessarily in $L^p$ with $p\not=2$.
     On the other hand, the following two norms are equivalent: for $1<p<\infty$,
$$
   \begin{cases}
      &(\|\psi\|_p^p + \|\nabla \psi\|_p^p)^{1/p}, \\
      &\|\sqrt{1 - \Delta}\,\psi\|_p.
   \end{cases}
$$
     where $\psi$ is a scalar-valued function in ${\mbbR}^3$ (Stein \cite{Stein},
     p.135, Theorem 3 or p.136, Lemma 3). However, for $p = 1$ or $p = \infty$, these
     two norms are not equivalent, in fact, the one does not dominate the other
     (\cite{Stein}, p.160, 6.6).

        For  an open subset $\Omega$ of $\mbbR^3$ and $1 \le p < \infty$, let
     $\Apio$ be all $C^{\infty}$ functions $\psi$ on $\Omega$ such that $\psi$ and
     $\nabla\psi$ belong to $L^p(\Omega)$ and let $H^{1,p}(\Omega)$ be the completion
     of $\Apio$ with respect to the norm given by
\begin{equation}\label{1sob}                                           
         \|\psi\|_{1,p,\Omega}
          = \left\{ \int_{\Omega} (|{\psi}(x)|^p +
             |\nabla{\psi}(x)|^p\,) \, dx \right\}^{1/p}
               = \{ \|\psi\|_{p,\Omega}^p + \|\nabla{\psi}\|_{p,\Omega}^p \}^{1/p},
\end{equation}
     where $|\nabla\psi(x)|^p = \sum_{j=1}^3 |\partial_j\psi(x)|^p$. Let $\Cono$ be
     the space of all $C^{\infty}$ functions $\phi$ on $\Omega$ such
     that support of $\phi$ is contained in $\Omega$ and let $H_0^{1,p}(\Omega)$ be the
     completion of $\Cono$ with respect to the norm (\ref{1sob}). The space
     $H_0^{1,p}(\Omega)$ is a closed subspace of $H^{1,p}(\Omega)$. The norm
     $\| \cdot \|_{S,1,p,\Omega}$ of the Banach space $[H^{1,p}(\Omega)]^4$ (and
     $[H_0^{1,p}(\Omega)]^4$) is given by
\begin{equation}\label{1SOB}                                           
       \|{\mbf}\|_{S,1,p,\Omega} =
          \left\{\int_{\Omega} \big(|{\mbf}(x)|_p^p
                 +|\nabla{\mbf}(x)|_p^p \,\big)\, dx\right\}^{1/p},
\end{equation}
     where ${\mbf} = {}^t(f_1, f_2, f_3, f_4) \in [H^{1,p}(\Omega)]^4$
     and
\begin{equation}\label{1Sobnorm}                                      
    \begin{cases}
      \ds   |{\mbf}(x)|_p^p =\sum_{k=1}^4 |f_k(x)|^p, \\
       \ds  |\nabla{\mbf}(x)|_p^p
                   = \sum_{j=1}^3 |\pa_j{\mbf}(x)|_p^p
                   = \sum_{j=1}^3 \sum_{k=1}^4 |\pa_j\mbf_k(x)|^p \\
   \end{cases}
\end{equation}
     (cf. (\ref{1gpnorm})).

\begin{defn}                                                   
        {\rm Let $\ApDio$ be all $[C^{\infty}(\Omega)]^4$ functions $f$ on $\Omega$ such that
     $f$ and $(\alpha \cdot \pp)f$ belong to $[L^p(\Omega)]^4$. Then the
     Dirac-Sobolev spaces $\mbbO$ and $\mbbHO$ are the completion of $\ApDio$ and
     $[\Cono]^4$ with respect to the norm
\begin{equation}\label{1S1p}                                        
       \|{\mbf}\|_{D,1,p,\Omega}
          = \left\{ \int_{\Omega} (|{\mbf}(x)|_p^p
                +|(\alpha \cdot \pp){\mbf}(x)|_p^p\,) \, dx \right\}^{1/p},
\end{equation}
     respectively, where ${\mbf}(x) = {}^t(f_1(x), f_2(x), f_3(x), f_4(x))$,
     and }
\begin{equation}\label{1nablaf}                                      
   \begin{cases}
      \ds    |{\mbf}(x)|_p^p = \sum_{k=1}^4 |f_k(x)|^p\,, \\
      \ds    |(\alpha \cdot \pp){\mbf}(x)|_p^p
                = |\sum_{j=1}^3 \alpha_j\pp_j{\mbf}(x)|_p^p
                = \sum_{k=1}^4\Big|\Big(\sum_{j=1}^3 -i\alpha_j\partial_j {f}\Big)_k(x)\Big|_p^p\,
   \end{cases}
\end{equation}
\end{defn}

\SP

        It should be noted that in the paper \cite{BES} the space ${\mathbb H}_0^{1,1}(\Omega)$
     of our paper was denoted by ${\mathbb H}^{1,1}(\Omega)$ without subscript `0'.
     We have adopted this notation, following the usual Sobolev space convention.

\SP

\begin{rem}                                                            
        {\rm (i) As in the case of Sobolev spaces, we have $\mbbHn = \mbbH$
     since $[\Con]^4$ is dense in $\mathcal{A}_{p,D}(\mbR^3)$ with respect to the norm (\ref{1S1p}).

\SP

        (ii) Let $\WpO$ be defined by
\begin{equation*}
           \WpO = \{ f \in [L^p(\Omega)]^4 \,;\, (\alpha\cdot \pp)f \in [L^p(\Omega)]^4 \},
\end{equation*}
     where $(\alpha\cdot \pp)f$ is taken in the sense of distributions. As in the case of
     Sobolev spaces (see, e.g., Adams-Fournier\,\cite{AF}, Theorem 3.17), by approximating
     elements in $\WpO$ using the molifier, we have  $\WpO = \mbbO$, where
     $1 \le p < \infty$ and $\Omega$ is an open subset of
     $\mbR^3$ . }
\end{rem}

        In this work we are going to investigate the relationship of the
     Dirac--Sobolev spaces $\mbbO$ and the ordinary Sobolev
     spaces $[H^{1,p}({\Omega})]^4$ as well as the relationship of the
     Dirac--Sobolev spaces $\mbbHO$ and the ordinary Sobolev
     spaces $[H_0^{1,p}({\Omega})]^4$.

        To proceed, we note the inclusions $[H^{1,p}(\Omega)]^4 \subset {\mathbb H}^{1,p}(\Omega)$
     and $[H_0^{1,p}(\Omega)]^4 \subset {\mathbb H}_0^{1,p}(\Omega)$ to hold,
     which we shall see precisely later in the next section, Proposition 2.2. So for an open
     subset $\Omega$ of ${\mathbb R}^3$, define the linear map $J_{\Omega}$ and $J_{0,\Omega}$ by
\begin{equation}\label{1J}                                 
   \begin{cases}
       J_{\Omega} \,:\, [H^{1,p}(\Omega)]^4 \ni \mbf
                      \mapsto J_{\Omega}f = f \in \mbbO, \\
       J_{0,\Omega} \,:\, [H_0^{1,p}(\Omega)]^4 \ni \mbf
                      \mapsto J_{0,\Omega}f = f \in \mbbHO.
   \end{cases}
\end{equation}

        Our main result is as follows:
\begin{thm}                                                        
        Let $\Omega$ be an open subset of $\mbbR^3$ and let $J_{\Omega}$ and $J_{0,\Omega}$
     be as above.

\SP

        {\rm (i)} Then, for $1 \le p < \infty$, the map both $J_{\Omega}$ and
     $J_{0,\Omega}$ are one-to-one and continuous. The Sobolev space $[H_0^{1,p}(\Omega)]^4$
     is a dense subspace of the Dirac--Sobolev spaces $\mbbHO$.

\SP

        {\rm (ii)} Let $1 < p < \infty$. Then $[H_0^{1,p}(\Omega)]^4 = \mbbHO$,
  i.e., the map $J_{0,\Omega}$ is not
     only one-to-one and continuous, but also they are onto with continuous
inverse map    $J_{0,\Omega}^{-1}$.

\SP

        {\rm (iii)} For $p = 1$ the map neither $J_{\Omega}$ nor $J_{0,\Omega}$ are onto, or
     we have $[H^{1,1}(\Omega)]^4 \subsetneqq \mathbb{H}^{1,1}(\Omega)$ and
     $[H_0^{1,1}(\Omega)]^4 \subsetneqq \mathbb{H}_{0}^{1,1}(\Omega)$. For $p = 1$, the norms
     $\| \cdot \|_{D,1,1,\Omega}$ and $\| \cdot \|_{S,1,1,\Omega}$ of these two spaces are
     not equivalent; $\| \cdot \|_{D,1,1,\Omega}$ is dominated by $\| \cdot \|_{S,1,1,\Omega}$,
     but not conversely.
\end{thm}                                            

\SP

\begin{rem}                                                
        {\rm We don't know in (ii) whether or not it holds
     for a proper open subset $\Omega$  that $[H^{1,p}(\Omega)]^4 = \mbbO$,
     i.e., that the map $J_{\Omega}$ is onto with continuous
     inverse map  $J_{\Omega}^{-1}$ when $\Omega \subsetneqq \mbR^3$,
     although it holds by (ii) for
     $\Omega = {\mbR^3}$ that $[H^{1,p}(\mbR^3)]^4 = {\Bbb H}^{1,p}(\mbR^3)$, because this
     space coincides with $[H_0^{1,p}(\mbR^3)]^4 = {\Bbb H}_0^{1,p}(\mbR^3)$.}
\end{rem}                                                  

\SP

        For the proof we shall use a method of classical analysis rather
     than a subtle pseudo-differentical calculus, in particular,
     in the  case $p=1$.

        In Section 2 we shall prove Theorem 1.3, (i) (Proposition 2.2). In Section 3 we are
     going to give the proof of Theorem 1.3, (ii) by first dealing with $J_{\Omega}$ and
     then $J_{0,\Omega}$. Theorem 1.3, (iii), the case that $p = 1$, will be discussed and proved
     in Section 4.

\vskip 30pt

\section{{\bf Continuity of the map ${\mathbf J_{\Omega}}$}}                 

\SP

        In this section we are going to prove Theorem 1.3, (i). Let
     $\alpha_j$, $j = 1, 2, 3$, be the Dirac matrices given in (\ref{1alpha}).
     Then we have

\begin{lem}                                                        
        Let ${\mb a} = {}^t(a_1, a_2, a_3, a_4) \in {\mathbb C}^4$. Then,
     for $p \in [1, \infty)$,
\begin{equation}\label{2alap}                                        
       |\alpha_j{\mb a}|_p = |{\mb a}|_p \q (j = 1, 2, 3),
\end{equation}
     where the norm $| \cdot |_p$ is given by {\rm(\ref{1gpnorm})}.
\end{lem}                                                    

\begin{proof}                                               
        By the definition of $\alpha_1$ we have
\begin{equation*}
         \alpha_1{\mb a} = {}^t(a_4, a_3, a_2, a_1),
\end{equation*}
     and hence
\begin{equation*}
     |\alpha_1{\mb a}|_p^p = |a_4|^p + |a_3|^p + |a_2|^p + |a_1|^p
                                                      = |{\mb a}|_p^p\,.
\end{equation*}
     In quite a similar manner (\ref{2alap}) can be proved for $j = 2, 3$.
\end{proof}                                           

        Now we are in a position to show the continuity of the map $J_{\Omega}$ and
     $J_{0,\Omega}$ given by
     (\ref{1J}).

\begin{prop}                                            
        Let $\Omega$ be an open subset of $\mbbR^3$ and let
     ${\mb f} \in [H^{1,p}(\Omega)]^4$. Then, for $p \in [1, \infty)$,
     we have ${\mb f} \in \mbbO$ and there exists a positive constant $C = C_p$,
     depending only on $p$, not on $\Omega$, such that
\begin{equation}\label{2est}                                                
       \|{\mb f}\|_{D,1,p,\Omega} \le C\|{\mb f}\|_{S,1,p,\Omega},
\end{equation}
     where the norms $\|{\mb f}\|_{D, 1,p,\Omega}$ and $\|{\mb f}\|_{S,1,p,\Omega}$
     are given in {\rm (\ref{1S1p})} and {\rm (\ref{1SOB})}, respectively. Thus
     the identity maps $J_{\Omega}$ on $[H^{1,p}(\Omega)]^4$ and
     $J_{0,\Omega}$ on $[H_0^{1,p}(\Omega)]^4$ are continuous, one-to-one maps from
     $[H^{1,p}(\Omega)]^4$ into $\mbbO$, and $[H_0^{1,p}(\Omega)]^4$ into $\mbbHO$.
     Further $[H_0^{1,p}(\Omega)]^4$ is a dense subset of $\mbbHO$.
\end{prop}                                                

\begin{proof}                                             
        Let ${\mb f} = {}^t(f_1, f_2, f_3, f_4) \in [H^{1,p}(\Omega)]^4$.
     By using Lemma 2.1 and H\"older's inequality for $p > 1$ or the triangle
     inequality for $p = 1$, we have
$$
\aligned
      |(\alpha\cdot \pp){\mb f}|_p^p &
            = |(\sum_{l=1}^3 -\alpha_l i\partial_l){\mb f}|_p^p
              \leq \big(\sum_{l=1}^3 |\alpha_l i\partial_l{\mb f}|_p\big)^p
            = \big(\sum_{\ell=1}^3 |\partial_l{\mb f}|_p\big)^p\\
      & \leq \big[\big(\sum_{l=1}^3 1^q\big)^{1/q}
              \big(\sum_{\ell=1}^3 |\partial_l{\mb f}|_p^p\,\big)^{1/p}\big]^p
            = 3^{p-1}|\nabla {\mb f}|_p^p\,,
\endaligned
$$
     where $p^{-1} + q^{-1} = 1$ and see (\ref{1Sobnorm}) for the definition of
     $|\nabla {\mb f}|_p$. It follows that
\begin{equation} \label{2acpfest}                                   
       \|(\alpha\cdot \pp){\mb f}\|_{p,\Omega}
           = \Big(\int_{\Omega} |(\alpha\cdot \pp){\mb f}|_p^p \, dx\Big)^{1/p}
           \leq 3^{\frac{p-1}{p}}\Big(\int_{\Omega} |\nabla {\mb f}|_p^p
                                                         \, dx\Big)^{1/p}
           = 3^{\frac{p-1}{p}}\|\nabla {\mb f}\|_{p,\Omega}\,,
\end{equation}
     where $\|(\alpha\cdot \pp){\mb f}\|_{p,\Omega}$ is the norm
     of $(\alpha\cdot \pp){\mb f} \in [L^p(\Omega)]^4$ given by (\ref{1pOnorm}),
     and $\|\nabla {\mb f}\|_{p,\Omega}$ is given by
\begin{equation*}
      \|\nabla {\mb f}\|_{p,\Omega}
       = \left\{ \int_{\Omega} |\nabla {\mb f}|_{p}^p \, dx \right\}^{1/p}\,.
\end{equation*}
     Then it is easy to see that (\ref{2acpfest}) implies (\ref{2est}).
     The map $J$ is one-to-one since, for ${\mb f}_j \in [H^{1,p}({\Omega})]^4$,
     $j = 1, 2$, we have
\begin{equation}\label{2oneone}                                            
      J_{\Omega}{\mb f}_1 = J_{\Omega}{\mb f}_2
                 \ \ {\rm in} \ \mbbO
         \Longrightarrow {\mb f}_1 = {\mb f}_2 \ \ {\rm in} \
                                                   [L^p(\Omega)]^4
      \Longrightarrow {\mb f}_1 = {\mb f}_2 \ \ {\rm in} \
                                                      [H^{1,p}(\Omega)]^4.
\end{equation}
     Using (\ref{2est}) and proceeding as in (\ref{2oneone}), we see that
     the identity map $J_{0,\Omega}$ on $[H_0^{1,p}(\Omega)]^4$ is also continuous and
     one-to-one on $[H_0^{1,p}(\Omega)]^4$. Since $[C_0^{\infty}(\Omega)]^4$ is dense
     in both $[H_0^{1,p}(\Omega)]^4$ and $\mathbb{H}_{0}^{1,p}(\Omega)$,
     $[H_0^{1,p}(\Omega)]^4$ is a dense subset of
$\mathbb{H}_{0}^{1,p}(\Omega)$. This completes the proof.
\end{proof}                                    

\vskip 30pt

\section{{\bf Range of the map ${\mathbf J_{0,\Omega}}$}}             

\SP

        In this section we are going to prove Theorem 1.3, (ii).

\begin{prop}                                                    
        Let $1 < p < \infty$.  Then the map
     $J_{0, \mbR^3}$ is onto $\mathbb{H}_0^{1,p}(\mbR^3)$.
\end{prop}                                                

        The proof will be given after the following two lemmas.

\begin{lem}                                                          
        Let $1 < q < \infty$. Then $\Delta(C_0^{\infty}(\mbR^3))$ is dense
     in $L^q(\mbbR^3)$.
\end{lem}                                                   

\begin{proof} [Proof of Lemma 3.2]                         
        Suppose that  $f \in L^r(\mbbR^3)$
     with $1/q +1/r =1$ satisfies
\begin{equation*}
     \langle f,\Delta \phi \rangle = \int f(x) \Delta \phi(x) dx =0,
          \quad \text{for all} \,\, \phi \in C_0^{\infty}(\mbR^3).
\end{equation*}
     Then we obtain in the sense of distributions $\Delta f = 0$, namely, $\Delta$
     annihilates $f$. By elliptic regularity, we see that $f$ must be $C^{\infty}$, and  hence $f(x)$ is a polynomial of $x$.
     Since $f(x)$ should belong to $ L^r(\mbbR^3)$, we have $f = 0$. 
This proves Lemma 3.2.
\end{proof}                                  

\begin{rem}                                                      
           {\rm Lemma 3.2 does not hold for $q = 1$, since,
 in this case, the Laplacian $\Delta$ always annihilates  a constant $C \ne 0$
which is a nonzero element of $L^{\infty}(\mbR^3) = L^1(\mbR^3)^*$.}
\end{rem}                                                    
\begin{lem}                                                                      
        Let $1 < q < \infty$. Let $\Omega \subset \mbR^3$ be an open set. Then, for each pair
     $(j, k)$, $j, k = 1, 2, 3$, there exists a positive constant $C = C_{jk}$ such that
\begin{equation}\label{3jkdeltain}                                        
        \|\pa_j\pa_k\phi \|_{q,\Omega} \le C\|\Delta \phi \|_{q,\Omega}
                                       \qquad (\phi \in C_0^{\infty}(\Omega)).
\end{equation}
\end{lem}                                                                

\begin{proof}                                                            
     By Stein \cite{Stein}, p.59, Proposition 3, there exists a positive
constant      $C = C_{jk}$ such that
\begin{equation*}
     \|\partial_j\partial_k\phi\|_q \leq C\|\Delta\phi\|_q,
                               \qquad \phi \in C_0^{\infty}({\mbbR}^3).
\end{equation*}
     Of course, this holds for $\phi \in C_0^{\infty}({\mbbR}^3)$.
\end{proof}                                                       

\begin{proof} [Proof of Proposition 3.1]                 
        The proof will be divided into four steps.

\SP

        (I) Let $f = {}^t(f_1,f_2,f_3,f_4) \in \mathbb{H}_0^{1,p}({\mbbR}^3)$.
     Then ${f} \in [L^p({\mbbR}^3)]^4$, and
\begin{equation*}
      g:= (\alpha\cdot \pp){ f}
       = -i[\alpha_1\partial_1 { f} +\alpha_2\partial_2 { f}
           +\alpha_3\partial_3 { f}]
\end{equation*}
     belongs to $[L^p({\mbbR}^3)]^4$.
       Using the definition (1.2) 
     of the Dirac matrices $\alpha_j$, $j = 1, 2, 3$, we can rewrite this with
     ${ g}={}^t(g_1,g_2,g_3.g_4)$ as
\begin{equation}\label{3gf}                                           
  \begin{cases}
      ig_1 = (\partial_1-i\partial_2)f_4+\partial_3f_3, \hskip 80pt \\
      ig_2 = (\partial_1+i\partial_2)f_3-\partial_3f_4, \\
      ig_3 = (\partial_1-i\partial_2)f_2+\partial_3f_1, \\
      ig_4 = (\partial_1+i\partial_2)f_1-\partial_3f_2.
  \end{cases}
\end{equation}
     Then from the first and second equations of (\ref{3gf}) we have
\begin{equation*}
   \aligned
     (\partial_1+i\partial_2)ig_1 &= (\partial_1^2+\partial_2^2)f_4
           + \partial_3(\partial_1+i\partial_2)f_3 \\
     & = (\partial_1^2+\partial_2^2)f_4 + \partial_3(ig_2+ \partial_3f_4),
   \endaligned
\end{equation*}
     so that
\begin{equation*}
     \Delta f_4 = (\partial_1^2+\partial_2^2+\partial_3^3)f_4
               = (\partial_1+i\partial_2)(ig_1)- \partial_3(ig_2),
\end{equation*}
     and hence, by applying $\pa_j$ to both sides of the above equation,
     we have, for $j = 1, 2, 3$,
\begin{equation}\label{3Del4}                                  
    \Delta\pa_j f_4 = (\pa_j\partial_1 + i\pa_j\partial_2)(ig_1)
                                         - \pa_j\partial_3(ig_2).
\end{equation}
     Similarly we have from (\ref{3gf})
\begin{equation}\label{3Del321}                                       
  \begin{cases}
       \Delta\pa_j f_3 = (\pa_j\partial_1 - i\pa_j\partial_2)(ig_2)
                                            - \pa_j\partial_3(ig_1), \\
       \Delta\pa_j f_2 = (\pa_j\partial_1 + i\pa_j\partial_2)(ig_3)
                                                - \pa_j\partial_3(ig_4), \\
       \Delta\pa_j f_1 = (\pa_j\partial_1 - i\pa_j\partial_2)(ig_4) -
                                                      \pa_j\partial_3(ig_3).
  \end{cases}
\end{equation}
     The equalities in (\ref{3Del4}) and (\ref{3Del321}) should be interpreted
     as equalities in the space ${\mathcal D}'({\mbbR}^3)$ of distributions
     on ${\mbbR}^3$.

\SP

        (II) Our first goal is to show that each distribution $\partial_jf_k$ actually
     belongs to $L^p({\mbbR}^3)$, where $j = 1, 2, 3$ and $k = 1, 2, 3, 4$, namely,
     for each $j$ and $k$ there exists $F_{jk} \in L^p({\mbbR}^3)$ such that
\begin{equation}\label{3f4}                                            
         \<\partial_jf_k, \ \phi\> = \int_{{\mbbR}^3} F_{jk}(x)\phi(x)\, dx
\end{equation}
     for any $\phi \in C_0^{\infty}({\mbbR}^3)$, where the left-hand side is a  bilinear form on
     ${\mathcal D}'({\mbbR}^3) \times C_0^{\infty}({\mbbR}^3)$.
This will show that $f$ belongs to $[H^{1,p}({\mbbR}^3)]^4$.
We shall prove (\ref{3f4}) for $k = 4$
     and $j = 1, 2, 3$ since other cases can be proved in a similar manner.
     After that, finally we show that $f$ belongs to
$[H_0^{1,p}({\mbbR}^3)]^4$ to complete our proof.

\SP

       (III) Let $q$ be the conjugate of $p$ or let $q$ satisfy $p^{-1} + q^{-1} = 1$.
     We see from (\ref{3Del4}) that for $\phi \in \Cono$,
\begin{eqnarray*}
     \langle \pa_j f_4, \Delta \phi \rangle
        &=& \langle \Delta \pa_j f_4, \phi \rangle
                        = \langle (\pa_j\partial_1 + i\pa_j\partial_2)(ig_1)
                           - \pa_j\partial_3(ig_2), \phi \rangle\\
        &=& \langle ig_1, (\pa_j\partial_1 + i\pa_j\partial_2)\phi\rangle
                  - \langle ig_2, \pa_j\partial_3 \phi \rangle.
\end{eqnarray*}
     Hence by Lemma 3.5 we have
\begin{multline}\label{3ffour}                                            
  \hskip 40pt  |\langle \pa_j f_4, \Delta \phi \rangle|
            \le \|g_1\|_p \|(\pa_j\partial_1 + i\pa_j\partial_2)\phi\|_q
              + \|g_2\|_p \|\pa_j\partial_3 \phi\|_q  \\
           \hskip 63pt \le (C_{j1}+C_{j2}) \|g_1\|_p \|\Delta \phi\|_q
              + C_{j3} \|g_2\|_p \|\Delta \phi\|_q   \\
              = [(C_{j1} + C_{j2}) \|g_1\|_p + C_{j3} \|g_2\|_p] \|\Delta \phi\|_q. \hskip 82pt
\end{multline}
     Since $\Delta C_0^{\infty}({\mbbR}^3))$ is dense in $L^q({\mbbR}^3)$ as has been shown in
     Lemma 3.2, the inequality (\ref{3ffour}) is extended uniquely to a continuous linear form
     on $L^q({\mbbR}^3)$. Since $L^p({\mbbR}^3)$ is the dual space of $L^q({\mbbR}^3)$, there
     exists a function $F_{j4} \in L^p({\mbbR}^3)$ such that
     $\<\pa_j f_4, \ \psi\> = \int_{{\mbbR}^3} F_{j4}(x)\psi(x) \, dx\,, \, \psi \in L^q({\mbbR}^3)$,
     which implies (\ref{3f4}) with $k = 4$ and $j = 1, 2, 3$.
     In particular, we have also shown that
\begin{equation}\label{3.7}
   \|\partial_jf_k \|_p \leq C_0 \{\Sigma_{k=1}^4\|g_k\|_p^p\}^{1/p}
  = C_0\|g\|_p,
\end{equation}
with a positive constant $C_0$ for all $j=1,2,3$ and $k=1,2,3,4$.

\SP

       (IV) Finally, since $f$ is $\mathbb{H}_0^{1,p}({\mbbR}^3)$,
     by definition there exist a sequence $\{f_n\}_{n=1}^{\infty}$
     in $C_0^{\infty}({\mbbR}^3)$ with $f_n =(f_{n,1},f_{n,2},f_{n,3},f_{n,4})$
     such that with
     $g_n =(g_{n,1},g_{n,2},g_{n,3},g_{n,4}) := (\alpha\cdot p) f_n$,
\begin{multline*}
   \hskip 50pt \|f_n -f\|_{D,1,p,{\mbbR}^3}^p = \|f_n -f\|_p^p
            + \|(\alpha\cdot \pp) (f_n -f)\|_p^p \\
            = \|f_n -f\|_p^p + \|g_n -g\|_p^p \rightarrow 0,
                               \quad\quad n \rightarrow \infty. \hskip 55pt
\end{multline*}
     Since by the same argument used to get (\ref{3.7})
     we have $\|\partial_j (f_{n,k}-f_k)\|_p \leq C_0\|g_n-g\|_p$
     for all $j=1,2,3$ and $k=1,2,3,4$, it follows that
     $\|f_n -f\|_{S,1,p,{\mbbR}^3}^p \rightarrow 0$
     as $n \rightarrow \infty$, so that $f \in [H_0^{1,p}({\mbbR}^3)]^4$.
     This completes  the proof   of Proposition 3.1.
\end{proof}                                    

\begin{prop}                                             
        Let $1 < p < \infty$. Let $\Omega \subset \mbR^3$ be an open set and $J_{0,\Omega}$
     be given in {\rm (\ref{1J})}. Then the map $J_{0,\Omega}$ is onto $\mathbb{H}_0^{1,p}(\Omega)$.
     Further, the inverse map $J_{0,\Omega}^{-1}$ is well-defined as a bounded linear
     operator.
\end{prop}                                             

\begin{proof}                                          
        (I) We have seen in Propositions 3.1 that
     $[H_0^{1,p}(\mbbR^3)]^4 = \mathbb{H}_0^{1,p}({\mbbR}^3)$ as sets and there
     exist positive constants $C_1$ and $C_2$ such that
\begin{equation}\label{3th12}                                        
          C_1\|{\mb f}\|_{D,1,p} \le \|{\mb f}\|_{S,1,p}
                                                   \le C_2\|{\mb f}\|_{D,1,p}
\end{equation}
     for ${\mb f} \in [H_0^{1,p}(\mbbR^3)]^4 = \mathbb{H}_0^{1,p}({\mbbR}^3)$.

\SP

        (II) Let ${\mb f} \in [H_0^{1,p}(\Omega)]^4$. Then there exists a sequence
     $\{\mbP_n\}_{n=1}^{\infty} \subset [\Cono]^4$ such that
\begin{equation}\label{3fPhiO}                                        
      \|{\mb f} - {\mbP}_n\|_{S,1,p,\Omega} \to 0 \qquad (n \to \infty).
\end{equation}
     Since each $\phi_n$ can be naturally extended to be an element of
     $[\Con]^4$ by setting ${\mbP}_n(x) = 0$ for
     $x \in \mbbR^3 \setminus \Omega$, we have
\begin{equation}\label{3fPhiR}                                        
      \|{\mb f} - {\mbP}_n\|_{S,1,p} \to 0 \qquad (n \to \infty),
\end{equation}
     where ${\mb f}$ is also extended to be a function on $\mbbR^3$ by setting
     $0$ outside $\Omega$, and hence ${\mb f} \in [H^{1,p}(\mbbR^3)]^4$ with
     support in the closure of $\Omega$. Therefore
     ${\mb f} \in \mathbb{H}^{1,p}({\mbbR}^3)$ and ${\mb f}$ satisfies
     (\ref{3th12}). Then, (\ref{3fPhiR}) is combined with (\ref{3th12}) to yield
\begin{equation}\label{3fPhiDR}                                        
      \|{\mb f} - {\mbP}_n\|_{D,1,p} \to 0 \qquad (n \to \infty),
\end{equation}
     which implies, together with the fact that and $\phi_n$ have support in
     $\Omega$, that
\begin{equation}\label{3fPhiDO}                                        
      \|{\mb f} - {\mbP}_n\|_{D,1,p,\Omega} \to 0 \qquad (n \to \infty).
\end{equation}
     Thus we have ${\mb f} \in \mathbb{H}_{0}^{1,p}(\Omega)$, and we
     obtain from (\ref{3th12})
\begin{equation}\label{3thO12}                                        
          C_1\|{\mb f}\|_{D,1,p,\Omega} \le \|{\mb f}\|_{S,1,p,\Omega}
                                         \le C_2\|{\mb f}\|_{D,1,p,\Omega}.
\end{equation}

        (III) Let ${\mb f} \in \mathbb{H}_{0}^{1,p}(\Omega)$. Then, starting with
     ${\mb f} \in \mathbb{H}_{0}^{1,p}(\Omega)$ proceeding as in (II), we can show that
     ${\mb f} \in [H_0^{1,p}(\Omega)]^4$ and the estimates $(\ref{3thO12})$ are
     satisfied, which completes the proof.
\end{proof}                                    

\begin{proof} [Proof of Theorem 1.3, (ii)]                     
        Theorem 1.3, (ii) follows from Propositions 3.1 and 3.5.
\end{proof}                                               

\vskip 30pt

\section{\bf The case ${\mb p = 1}$}

\SP

        The goal of this section is to prove Theorem 1.3, (iii), that is, to prove
     $[H^{1,1}(\Omega)]^4$ and $[H_0^{1,1}(\Omega)]^4$ are proper subspaces of
     $\mathbb{H}^{1,1}(\Omega)$ and $\mathbb{H}_{0}^{1,1}(\Omega)$, respectively.
     First, we are going to show, for $\Omega = {\mathbb R}^3$, that
     $[H_0^{1,1}({\mathbb R}^3)]^4 = [H^{1,1} ({\mathbb R}^3)]^4$ is a  proper
     subspace of ${\mathbb H}_{0}^{1,1}({\mathbb R}^3) = {\mathbb H}^{1,1}({\mathbb R}^3)$
     (Proposition 4.4). Then all other statements in Theorem 1.3, (iii) will follow from
     Proposition 4.4. In the following, when speaking of ${\mathbb H}_{0}^{1,1}({\mathbb R}^3)$
     or ${\mathbb H}^{1,1}({\mathbb R}^3)$, and $[H_0^{1,1}({\mathbb R}^3)]^4$ or
     $[H^{1,1}({\mathbb R}^3)]^4$, we shall use the latter, namely,
     ${\mathbb H}^{1,1}({\mathbb R}^3)$ and $[H^{1,1}({\mathbb R}^3)]^4$.  As easily seen,
     ${\mathbb H}^{1,p}(\mathbb R^3)$ is also  the subspace of $[L^p({\mathbb R}^3)]^4$
     consisting of all ${f} \in [L^p({\mathbb R}^3)]^4$ such that
     $(\alpha\cdot \pp + \beta) f \in [L^p({\mathbb R}^3)]^4$ instead of
     $(\alpha\cdot \pp) f \in [L^p({\mathbb R}^3)]^4$, where $\beta$ is the fourth
     Dirac matrix $\beta$ given by (\ref{1betam}).

\SP

\begin{lem}                                                 
        The map
\begin{equation*}
            (\alpha\cdot \pp) + \beta \,:\, \mathbb{H}^{1,1}(\mbR^3) \ni f
            \mapsto (\alpha\cdot \pp + \beta)f \in [L^1(\mbR^3)]^4
\end{equation*}
        maps $\mathbb{H}^{1,1}(\mbR^3)$ one-to-one and onto $[L^1(\mbR^3)]^4$.
\end{lem}                                                         

\begin{proof}                                                    
        (I) We define the Dirac operator $H_0 = (\alpha\cdot \pp) + \beta$ as a linear
     operator in $[L^1(\mbR^3)]^4$ with domain $D(H_0) = \mathbb{H}^{1,1}(\mbR^3)$.
     It is easy to see that $H_0$ is a closed operator in $[L^1(\mbR^3)]^4$. Let
     the operator $H = (\alpha\cdot \pp) + \beta$ be defined as a pseudodifferential
     operator acting on $[\cS'(\mbbR^3)]^4$, the dual space of $[\cS(\mbbR^3)]^4$,
     with $4\times 4$ matrix symbol
\begin{equation}                                                       
         \sigma_H(\xi) = \alpha\cdot\xi + \beta
                       = \sum_{j=1}^3 \xi_j\alpha_j + \beta.
\end{equation}
     Then the operator $H_0$ can be viewed as the restriction of the operator $H$
     to $\mathbb{H}^{1,1}(\mbR^3)$. Let $B$ be a pseudodifferential operator acting
     on $[\cS'(\mbbR^3)]^4$ with symbol
\begin{equation*}
      \sigma_{B}(\xi) = (1 + |\xi|^2)^{-1}\sigma_{H}(\xi)
                               = \sigma_H(\xi)[(1 + |\xi|^2)^{-1}I_4]. \hskip 50pt
\end{equation*}
     By the anti-commutative relation
\begin{equation*}
          \alpha_j\alpha_k + \alpha_k\alpha_j = 2\delta_{jk}I_4
                               \qquad (j, k = 1, 2, 3, 4, \alpha_4 = \beta),
\end{equation*}
     where $I_4$ is the $4 \times 4$ unit matrix, we see that
\begin{equation*}
          \sigma_{H}(\xi)\sigma_{B}(\xi)
                                = \sigma_{B}(\xi)\sigma_{H}(\xi) = I_4,
\end{equation*}
     which implies that
\begin{equation}\label{4AB}                                       
       HB\mbf = BH\mbf = \mbf \qquad
                                     (\mbf \in [\cS'(\mbbR^3)]^4),
\end{equation}
     {\it i.e.,} the operator $B$ is the inverse operator of $H$ on $[\cS'(\mbbR^3)]^4$.

\SP

        (II) Note that
\begin{equation}\label{4B0dec}                                    
                B = HB^{(1)},  \hskip 150pt
\end{equation}
     where the symbol $\sigma_{B^{(1)}}(\xi)$ of $B^{(1)}$ is given by
     $\sigma_B^{(1)}(\xi) = (|\xi|^2 + 1)^{-1}I_4$. Let
     $\sigma_{0}(\xi) = (|\xi|^2 + 1)^{-1}$ and let $\sigma_{0}(\pp)$ be the pseudodifferential
     operator on $\cS'(\mbbR^3)$ with symbol $\sigma_{0}(\xi)$. The symbol $\sigma_{0}(\xi)$
     is a $C^{\infty}$ function on $\mbbR_{\xi}^3$ and bounded together with all their
     derivatives. Then, by noting that the integrand $(|\xi|^2 + 1)^{-1}(\cF\phi)(\xi)$
     is a function in $\cS(\mbbR_{\xi}^3)$ for $\phi \in \cS(\mbbR^3)$, we have
\begin{multline*}
   \hskip 50pt (\sigma_0(\pp)\phi)(x) = \lim_{R\to\infty}
        (2\pi)^{-3/2}\int_{|\xi|<R} e^{ix\cdot\xi}(|\xi|^2 + 1)^{-1}(\cF\phi)(\xi)\,d\xi \\
            = (2\pi)^{-3}\lim_{R\to\infty} \int_{\mbbR^3}
                \Big\{\int_{|\xi|<R} e^{i(x - y)\cdot\xi}(|\xi|^2 + 1)^{-1} \,d\xi\Big\}
                   \phi(y) \, dy. \hskip 1pt \\
            = (2\pi)^{-3}\lim_{R\to\infty} \int_{\mbbR^3}
                \Big\{\int_{|\xi|<R} e^{iy\cdot\xi}(|\xi|^2 + 1)^{-1} \,d\xi\Big\}
                   \phi(x - y) \, dy. \hskip 1pt \\
\end{multline*}
     for $\phi \in \cS(\mbbR^3)$. Since $(|\xi|^2 + 1)^{-1} \in L^2(\mbbR_{\xi}^3)$, the
     integral $\int_{|\xi|<R} e^{iy\cdot\xi}(|\xi|^2 + 1)^{-1} \,d\xi$ converges in
     $L^2(\mbbR_y^3)$ as $R \to \infty$. At the same time it is known that the limit
\begin{equation*}
        \lim_{R\to\infty} (2\pi)^{-3}\int_{|\xi|<R} e^{iy\cdot\xi}(|\xi|^2 + 1)^{-1} \,d\xi
                                                                = G(y)
\end{equation*}
     exist for $y \ne 0$ with
\begin{equation*}
         G(y) =  \frac{e^{-|y|}}{4\pi|y|}, \hskip 80pt
\end{equation*}
     which is the Green function of the operator $1 - \Delta$. Thus we have
\begin{equation*}
         (\sigma_0(\pp)\phi)(x) = (G*\phi)(x) \qquad (\phi \in \cS(\mbbR^3)),
\end{equation*}
     where $G*\phi$ denotes the convolution of $G$ and $\phi$, which implies that
\begin{equation}\label{4convo}                                       
        \sigma_B^{(1)}(\pp)\mbP(x) = (G*\phi)(x)
                    = {}^t((G*\phi_1)(x), (G*\phi_2)(x), (G*\phi_3)(x), (G*\phi_4)(x))
\end{equation}
     for $\mbP = {}^t(\phi_1, \phi_2, \phi_3, \phi_4) \in [\cS(\mbbR^3)]^4$.

\SP

        (III) Let $B$ be the pseudodifferential operator as in (I). It follows from
     (\ref{4convo}) that
\begin{equation}\label{4B0}                                                   
      B\mbP(x) = \sigma_H(\pp)\sigma_B^{(1)}(\pp)\mbP(x)
                   = \Big(\sum_{j=1}^3 -i\alpha_j\pa_j + \beta\Big)(G*\mbP)(x)
\end{equation}
     for $\mbP(x) \in [\cS(\mbbR^3)]^4$, where $\pa_j = \pa/\pa x_j$.
     Define a $4 \times 4$ matrix-valued function $K(x) = (K_{k\ell}(x))_{1\le k,\ell\le4}$ on
     $\mbbR^3$ by
\begin{equation*}
               K(x) = \Big(\sum_{j=1}^3 -i\alpha_j\pa_j + \beta\Big){\mb G}(x)
                      \qquad ({\mb G}(x) = {}^t(G(x), G(x), G(x), G(x)).
\end{equation*}
     Therefore we have
\begin{equation}\label{4K(x)}                                                    
     K(x) = \frac1{4\pi}\Big[ \sum_{j=1}^3 i\alpha_j
                 \Big(\frac{x_j}{|x|^3} + \frac{x_j}{|x|^2}\Big)
                                   + \beta\frac1{|x|}\Big]e^{-|x|},
\end{equation}
     and each element $K_{k\ell}(x)$ of $K(x)$ belongs to $L^1(\mbbR^3)$. Thus
     the $k$-th component $(B\mbP)_k$, $k = 1, 2, 3, 4$, of $B\mbP$ is expressed as
\begin{equation*}
           (B\mbP)_k(x) = \sum_{\ell=1}^4 (K_{k\ell}*\phi_{\ell})(x),
\end{equation*}
     which allows us to apply Young's inequality to see that
\begin{equation*}
           \|B\mbP\|_1 \le C\|\mbP\|_1   \qquad (\mbP \in [\cS(\mbbR^3)]^4)
\end{equation*}
     with a positive constant $C$. Therefore $B$ restricted on $[\cS(\mbbR^3)]^4$ is
     uniquely extended to a bounded linear operator on $[L^1(\mbbR^3)]^4$ which will be denoted
     by $B_0$. The operator $B_0$ is actually the restriction of $B$ to $[L^1(\mbbR^3)]^4$.

\SP

        (IV) Let $\mbg \in [L^1(\mbbR^3)]^4$. Let $\{ \mbg_m \}_{m=1}^{\infty} \subset [\cS(\mbbR^3)]^4$
     be a sequence such that $\mbg_m \to \mbg$ in $[L^1(\mbbR^3)]^4$ as $m \to \infty$. It follows from
     (\ref{4AB}) that
\begin{equation*}
             H_0B_0\mbg_m = (\alpha\cdot\pp + \beta)B_0\mbg_m = \mbg_m
                                               \qquad (m = 1, 2, \cdots).
\end{equation*}
     Therefore, recalling that $B_0$ is a bounded operator on $[L^1(\mbbR^3)]^4$, we have
\begin{equation}\label{4Bmbfm}                                             
    \begin{cases}
       B_0\mbg_m \to B_0\mbg, \\
       (\alpha\cdot\pp + \beta)B_0\mbg_m = \mbg_m \to \mbg
    \end{cases}
\end{equation}
     in $[L^1(\mbbR^3)]^4$ as $m \to \infty$. Since the operator $H_0 = \alpha\cdot\pp + \beta$
     defined on ${\mathbb H}^{1,1}({\mbbR}^3)$ is a closed operator, we see
     from (\ref{4Bmbfm}) that $B_0\mbg \in {\mathbb H}^{1,1}({\mbbR}^3)$ and $H_0B_0\mbg = \mbg$,
     which implies that $H_0 = \alpha\cdot\pp + \beta$ is onto $[L^1(\mbbR^3)]^4$.

\SP

        (V) Suppose that $\mbf \in {\mathbb H}^{1,1}({\mbbR}^3)$ such that
     $H_0\mbf = (\alpha\cdot\pp + \beta)\mbf = 0$. Let
     $\{ \mbf_m \}_{m=1}^{\infty} \subset [\Con]^4$ such that $\mbf_m \to \mbf$
     in ${\mathbb H}^{1,1}({\mbbR}^3)$ as $m \to \infty$. Thus we have
\begin{equation}\label{4glim}                                                          
   \mbf_m \to \mbf, \quad
      (\alpha\cdot\pp + \beta)\mbf_m \to (\alpha\cdot\pp + \beta)\mbf \ \ \ {\rm in} \ \
                                                                      [L^1(\mbbR^3)]^4.
\end{equation}
     On the other hand, we have from (\ref{4AB})
\begin{equation}\label{4ident}                                                      
            B_0(\alpha\cdot\pp + \beta)\mbf_m = (\alpha\cdot\pp + \beta)B_0\mbf_m = \mbf_m
                                       \qquad (m = 1, 2, \cdots).
\end{equation}
     Letting $m \to \infty$ in (\ref{4ident}), noting that $B_0$ is a bounded operator
     and using (\ref{4glim}), we see that
\begin{equation*}
           0 = B_0(\alpha\cdot\pp + \beta)\mbf
                     = \lim_{m\to\infty} (\alpha\cdot\pp + \beta)B_0\mbf_m
                         = \lim_{m\to\infty} \mbf_m =  \mbf
\end{equation*}
     in $[L^1(\mbbR^3)]^4$, which implies that $H_0$ is one-to-one. This completes the proof
     of Lemma 4.1.
\end{proof}                                                    

\SP

\begin{rem}                                                             
        {\rm As has been seen, the key element of the proof of the above Lemma 4.1 is to
      show that the pseudodifferential operator $B$ given by (\ref{4B0dec}) is a
      bounded linear operator on $[L^1(\mbR^3)]^4$. Actually it can be shown that
      $B$ is a bounded linear operator on $[L^p(\mbR^3)]^4$ for $1 \le p < \infty$.
      Thus we can prove that Lemma 4.1 holds for any $1 \le p < \infty$. In fact,
      for $1 < p < \infty$, a theorem in Fefferman\,\cite{F} (Theorem, a), p.414)
      can be applied to show that $B$ is a bounded linear operator on $[L^p(\mbR^n)]^4$
      with $n = 3$. Let $\sigma(x, \pp)$ be a pseudodifferential operator in $\mbR^n$ whose
      symbol $\sigma(x, \xi)$ belongs to the H\"ormander class $S^{-b}_{1-a,\delta}({\mbR}^n)$
      with $0 \le \delta < 1 - a < 1$. Then it follows from the above theorem by Fefferman
      that $\sigma(x, \pp)$ is a bounded operator on $L^{p}(\mbR^n)$ if
      $b < na/2$ and
$$
     \left|\frac1{p}-\frac12\right| \le \frac{b}{n}
          \Big[\frac{\frac{n}{2} + \lambda}{b + \lambda}\Big]  \qquad \qquad
                        \Big(\ds \lambda = \frac{\frac{na}{2}-b}{1-a}\Big).
$$
      By taking $n = 3,\, b = 1,\, \delta = 0$, the above two condition becomes
\begin{equation}\label{4pineq}                                                   
           3\left|\frac1{p} - \frac1{2}\right| \le \frac1{a} < \frac32.
\end{equation}
      For $p > 1$ there exists $a \in (0, 1)$ which satisfies (\ref{4pineq}). For
      $p = 1$, however, there is no $a$ which satisfies (\ref{4pineq}) since
      both sides of (\ref{4pineq}) become 3/2. Indeed, our pseudodifferential operator
      $B$ has symbol $\sigma_B(\xi)$ belonging to the H\"ormander class
      $S^{-1}_{1-a,0}({\mathbb R}^3)$. To prove Lemma 4.1, which is the case
      $p = 1$, we have discussed the integral kernel of the Dirac operator. }
\end{rem}                                                         

        To proceed, we need some facts on the local Hardy space $h^1(\mbbR^3)$, which
     is introduced in Goldberg\,\cite{G} in connection with the  Hardy space $H^1(\mbbR^3)$. The
     Hardy space is (see e.g. Fefferman-Stein\, \cite{FS}) the proper subspace of $L^1({\mbbR}^3)$
     consisting of the functions $f \in L^1({\mbbR}^3)$ such that $R_j f \in L^1({\mbbR}^3)$ for
     $j=1, 2, 3$, where $R_j :=\partial_j \cdot (-\Delta)^{-1/2}$ are the Riesz
     transforms, having symbols $i\xi_j/|\xi|$. Let  $\varphi$ be a
     fixed function in the Schwartz space ${\cS(\mbbR^3)}$ such that $\varphi =1$ in a
     neighborhood of the origin. By definition a distribution $f$ belongs to $h^1({\mbbR}^3)$
     if and only if $f \in L^1({\mbbR}^3)$ and $r_j f \in L^1({\mbbR}^3)$ for $j=1,
     2, 3$, where $r_j,\, j=1, 2, 3$, are pseudodifferential operators with
     symbol $\sigma_{r_j}(\xi) = (1-\varphi(\xi))(i\xi_j/|\xi|)$ (\cite{G}, Theorem 2 (p.33)).
     The definition is independent of the choice of $\varphi$. It is a Banach
     space with norm $\|f\|_{h^1} = \|f\|_{L^1} + \sum_{j=1}^3 \|r_j f\|_{L^1}$.
     The space $h^1(\mbbR^3)$ is a proper subspace of $L^1({\mbbR}^3)$, which is strictly
     larger than the Hardy space $H^1({\mbbR}^3)$ (see e.g. \cite{G}, p.33, just after
     Theorem 3).

        Now, we are introducing the following operator
\begin{equation}\label{4r'jdef}                                             
   r'_j = \partial_j (1-\Delta)^{-1/2}
         = \frac{\partial_j}{(-\Delta)^{1/2}}
         \frac{(-\Delta)^{1/2}}{(1-\Delta)^{1/2}}
       =R_j\cdot \frac{(-\Delta)^{1/2}}{(1-\Delta)^{1/2}},
\end{equation}
     where we note that the pseudodifferential operator $(-\Delta)^{1/2}/(1-\Delta)^{1/2}$ is
     a bounded operator on $L^1({\mbbR}^3)$ (see Stein\,\cite{Stein}, p.133, Eq.(31)).

        The proof of the lemma below was inspired by the proof of \cite{G}, Theorem 2 (p.33).

\SP

\begin{lem}                                                                   
         A distribution $f$ in $\mbbR^3$ belongs to $h^1({\mbbR}^3)$ if and only if
    $f \in L^1({\mbbR}^3)$ and $r'_j f \in L^1({\mbR}^3)$ for $j= 1, 2, 3$.
\end{lem}                                                             

\begin{proof}                                                       
        (I) It is sufficient to show that $r_j - r'_j$,
     $j = 1, 2, 3$, are bounded linear operators on $L^1(\mbbR^3)$ (or, more exactly, the
     pseudodifferential operator $r_j - r_j'$ defined on $\cS(\mbbR^3)$ can be uniquely
     extended to a bounded linear operator on $L^1(\mbbR^3)$). Note that the operators
     $r_j$ and $r'_j$ have symbols
\begin{equation}\label{4rjr'j}                                                     
       \sigma_{r_j}(\xi) = \frac{(1-\varphi(\xi))i\xi_j}{|\xi|}, \quad
            \sigma_{r'_j}(\xi) = \frac{i\xi_j}{(1+|\xi|^2)^{1/2}}\,,
\end{equation}
     and both symbols are $C^{\infty}$ functions in ${\mathbb R}^3_{\xi}$ and bounded
     together with all their derivatives, and we have
\begin{multline*}
     \hskip 50pt \sigma_{r_j}(\xi) - \sigma_{r'_j}(\xi)
           = \frac{(1-\varphi(\xi))i\xi_j}{|\xi|}\Big(1-\frac{|\xi|}{(1+|\xi|^2)^{1/2}}\Big)
              - \frac{\varphi(\xi)i\xi_j}{(1+|\xi|^2)^{1/2}} \\
               =: \sigma_{1j}(\xi) + \sigma_{2j}(\xi). \hskip 0pt
\end{multline*}
     As in the proof of Lemma 4.1, we are going to show that, for each $j = 1, 2, 3$, the
     pseudodifferential operator with symbols $\sigma_{1j}$ and $\sigma_{2j}$ have integral kernels
     belonging to $L^1(\mbbR^3)$, in other words, that their inverse Fourier transforms
     $\coF\sigma_{1j}(x)$ and $\coF\sigma_{2j}(x)$ belong  to $L^1({\mbbR}^3)$, where $\coF$ is
     given by
\begin{equation*}
           \coF\phi(x) = (2\pi)^{-3/2}\int_{\mbbR^3} e^{ix\cdot\xi}\phi(\xi) \, d\xi.
\end{equation*}

        It is easy to see that $\coF\sigma_{2j} \in L^1({\mbbR}^3)$, because $\sigma_{2j}$
     belongs to $\cS(\mbbR^3)$, so that $\coF\sigma_{2j}$ belongs to $\cS(\mbbR^3)$ and hence it
     belongs to $L^1(\mbbR^3)$. In the rest of the proof we are going to show that
     $\coF\sigma_{1j} \in L^1({\mbbR}^3)$.

\SP

        (II) By definition we have
$$
      \sigma_{1j}(\xi) =
           i(1-\varphi(\xi))\frac{\xi_j}{|\xi|(1+|\xi|^2)^{1/2}(|\xi|+(1+|\xi|^2)^{1/2})},
$$
     and hence, $\sigma_{1j}(\xi) = O(|\xi|^{-2})$ as $|\xi| \to \infty$. Thus, by noting that
     $1-\varphi(\xi)$ is $0$ around the origin $\xi = 0$, we see that $\sigma_{1j} \in L^2(\mbbR^3)$.
     Therefore $I_j(x) = (\coF\sigma_{1j})(x)$ exists as a function in $L^2(\mbbR_x^3)$.
     Let $\rho(t)$ be a real-valued $C^{\infty}$ function on $[0, \infty)$ such that
\begin{equation*}
   \rho(t) = 1 \quad (0 \le t \le 1), \ \ \ = 0 \quad (t \ge 2).
\end{equation*}
     Then, since $\sigma_{1j}(\xi)\rho(\ep|\xi|)$, $\ep > 0$, converges to $\sigma_{1j}(\xi)$ in
     $L^2(\mbbR_{\xi}^3)$ as $\ep \downarrow 0$, we have, by setting
\begin{equation*}
           I_j(x, \ep) := \coF(\sigma_{1j}(\xi)\rho(\ep|\xi|)),
\end{equation*}
     $I_j(x, \ep)$ converges to $I_j(x)$ in $L^2(\mbbR_x^3)$ as $\ep \downarrow 0$,
     and hence there exists a decreasing sequence
\begin{equation*}
         1 \ge \ep_2 > \ep_2 > \cdots > \ep_m > \ \to 0
\end{equation*}
     such that $I_j(\ep_m, x) \to I_j(x)$ a.e. $x$ as $m \to \infty$. For the
     sake of the simplicity of notations, we shall use $\ep \le 1$ instead of $\ep_m$.

\SP

        (III) Let $\alpha = (\alpha_1, \alpha_2, \alpha_3)$ be a multi-index. Then
     we have
\begin{equation}\label{4ests1j}                                              
         |\pa_{\xi}^{\alpha}\sigma_{1j}(\xi)| \le C_{j,\alpha}(1 + |\xi|)^{-2-|\alpha|}
                                     \qquad (\xi \in \mbbR_{\xi}^3 )
\end{equation}
     with a constant $C_{j,\alpha} > 0$, where we should note that $1 - \varphi(\xi)$
     is bounded and
$$
        \pa_{\xi}^{\alpha}(1 - \varphi(\xi)) = - \pa_{\xi}^{\alpha}\varphi(\xi)
                                                       \in \cS(\mbbR_{\xi}^3)
$$
     for $\alpha \ne 0$. Let $\ell$ be a positive integer and $k = 1, 2, 3$. Then,
     by integration by parts,
\begin{multline}\label{4Jepxi}                                               
    \hskip 26pt \ds (2\pi)^{3/2}x_k^{\ell}I_j(x, \ep) = \int_{\mbbR^3}\big\{(-i\pa_{\xi_k}^{\ell})
                        e^{ix\cdot\xi}\big\}\sigma_{1j}(\xi)\rho(\ep|\xi|) \, d\xi \\
    \ds       = (-i)^{\ell}(-1)^{\ell}\int_{\mbbR^3} e^{ix\cdot\xi}
         \pa_{\xi_k}^{\ell}\big\{\sigma_{1j}(\xi)\rho(\ep|\xi|)\big\}\, d\xi. \hskip 48pt
\end{multline}
     Here, by the Leibniz formula, we have
\begin{multline*}
      \ds \pa_{\xi_k}^{\ell}\big\{\sigma_{1j}(\xi)\rho(\ep|\xi|)\big\}
            = \big(\pa_{\xi_k}^{\ell}\sigma_{1j}(\xi)\big)\rho(\ep|\xi|)
               + \sum_{m=1}^\ell {}_{\ell}C_m \big(\pa_{\xi_k}^{\ell-m}\sigma_{1j}(\xi)\big)
                       \big(\pa_{\xi_k}^{m}\rho(\ep|\xi|)\big) \\
      \ds   =: J_{0}(\xi, \ep) + J_1(\xi, \ep) \hskip 180pt
\end{multline*}
     with ${}_{\ell}C_m = \ell/(m!(\ell-m)!)$. For $m = 1, 2, \cdots, \ell$, we have
\begin{equation*}
     \xi \in \ {\rm supp}(\pa_{\xi_k}^{m}\rho(\ep|\xi|)) \Longrightarrow
                 1 \le \ep|\xi| \le 2 \Longrightarrow \ep \le \frac2{|\xi|}\,,
\end{equation*}
     where ${\rm supp}(f)$ denotes the support of $f$. Thus we can replace $\ep$ in
     $\pa_{\xi_k}^{m}\rho(\ep|\xi|)$ by $2|\xi|^{-1}$ when we evaluate
     $|\pa_{\xi_k}^{m}\rho(\ep|\xi|)|$. Therefore it follows that
\begin{equation}\label{4parho}                                               
         |\pa_{\xi_k}^{m}\rho(\ep|\xi|)| \le c(1 + |\xi|)^{-m}\chi_{\ep}(\xi),
\end{equation}
     where $c = c_{j,k,m}$ is a positive constant and $\chi_{\ep}(\xi)$ is the characteristic
     function of the set  $A_{\ep} = \{ \xi \,:\, \ep^{-1} \le |\xi| \le 2\ep^{-1} \}$.
     Since it is supposed that $\ep \le 1$, we have $A_{\ep} \subset \{ \xi \,:\, |\xi| \ge 1 \}$.
     The inequalities (\ref{4ests1j}) and (\ref{4parho}) are combined to give
\begin{equation*}
        |J_1(\xi, \ep)| \le C(1 + |\xi|)^{-2 - \ell}\chi_{\ep}(\xi) \qquad (\xi \in \mbbR_{\xi}^3)
\end{equation*}
     with positive constant $C = C_{j,k,\ell}$. Let $\ell \ge 2$. Then it is seen that
     $|J_1(\xi, \ep)|$ is dominated by $C(1 + |\xi|)^{-2 - \ell}$, which is in $L^1(\mbbR_{\xi}^3)$,
     and $J_1(\xi, \ep) \to 0$ for each $\xi \in \mbbR_{\xi}^3$ as $\ep \to 0$, and hence, by
     the Lebesgue convergence theorem, we have
\begin{equation*}
        \int_{\mbbR^3} e^{ix\cdot\xi}J_1(\xi, \ep) \, d\xi \to 0 \quad (\ep \to 0).
\end{equation*}
     Similarly, since $|J_{0}(\xi, \ep)|$ is dominated by $|\pa_{\xi_k}^{\ell}\sigma_{1j}(\xi)|$
     which is in $L^1(\mbbR_{\xi}^3)$, and $J_{0}(\xi, \ep)$ converges to
     $\pa_{\xi_k}^{\ell}\sigma_{1j}(\xi)$ for each $\xi \in \mbbR_{\xi}^3$ as $\ep \to 0$, we have
\begin{equation*}
        \int_{\mbbR^3} e^{ix\cdot\xi}J_0(\xi, \ep) \, d\xi \to \int_{\mbbR^3} e^{ix\cdot\xi}
                          \pa_{\xi_k}^{\ell}\sigma_{1j}(\xi) \, d\xi \quad (\ep \to 0).
\end{equation*}
     Therefore, by letting $\ep \to 0$ in (\ref{4Jepxi}), we obtain
\begin{equation}\label{4xlequal}                                               
        x_k^{\ell}I_j(x) = x_k^{\ell}(\coF\sigma_{1j})(x)
               = i^{\ell}(2\pi)^{-3/2}\int_{\mbbR^3} e^{ix\cdot\xi}
                          \pa_{\xi_k}^{\ell}\sigma_{1j}(\xi) \, d\xi
\end{equation}
     a.e. $x \in \mbbR^3$ for $\ell \ge 2$ and $j, k = 1, 2, 3$. Here the right-hand
     side is uniformly bounded for $x \in \mbbR^3$. Thus, by considering the case $\ell = 2$
     and $\ell = 4$, it follows that
\begin{equation*}
           |(\coF\sigma_{1j})(x)| \le C_j\min (|x|^{-2}, |x|^{-4}) \quad (j = 1, 2, 3)
\end{equation*}
     with a positive constant $C_j$, which implies that $\coF\sigma_{1j} \in L^1(\mbbR^3)$.
     This completes the proof of Lemma 4.3.
\end{proof}                                                     

\SP

        Now we are going to prove that $[H^{1,1}(\mbbR^3)]^4$ is a proper subspace of
     $\mathbb{H}^{1,1}(\mbbR^3)$, which is the most crucial part of Theorem 1.3,
     (iii) in the following strategy: Let $\mbg \in [L^1(\mbbR^3)]^4 \setminus [h^1(\mbbR^3)]^4$,
     where $h^1(\mbbR^3)$ is the local Hardy space which can be defined, by Lemma 4.3, as
     the space of all distributions $f$ such that $f \in L^1({\mbbR}^3)$ and $r'_j f \in L^1({\mbbR}^3)$
     for $j=1, 2, 3$, where $r'_j$ is given by (\ref{4r'jdef}). Set
     $\mbf = (\alpha\cdot\pp + \beta)^{-1}\mbg$. Then, by using Lemma 4.1,
     we have $\mbf \in \mathbb{H}^{1,1}(\mbbR^3)$. Then we shall be able to show that
     $\mbf \notin [H^{1,1}(\mbbR^3)]^4$.

\SP

\begin{prop}                                                           
        ${\mathbb H}^{1,1}(\mbbR^3)$ is strictly larger than $[H^{1,1}(\mbbR^3)]^4$.
\end{prop}                                                          

\begin{proof}                                                     
        (I) It follows from Lemma 4.1 that, for every $\mbg \in [L^1(\mbbR^3)]^4$ there exists
     a unique $\mbf \in {\mathbb H}^{1,1}(\mbbR^3)$ such that $\mbg = (\alpha\cdot \pp + \beta)\mbf$.
     The equation $\mbg = (\alpha\cdot\pp + \beta)\mbf$ can be written as
\begin{equation*}
    \begin{cases}
        &g_1= -i(\partial_1-i\partial_2)f_4 -i\partial_3f_3 +f_1,\\
        &g_2= -i(\partial_1+i\partial_2)f_3 +i\partial_3f_4 +f_2, \\
        &g_3= -i(\partial_1-i\partial_2)f_2 -i\partial_3f_1 -f_3,\\
        &g_4= -i(\partial_1+i\partial_2)f_1 +i\partial_3f_2 -f_4.\\
    \end{cases}
\end{equation*}
     Solving the above equation for $f_1, f_2, f_3$ and $f_4$, we obtain
\begin{equation}\label{4eqf}                                                
    \begin{cases}
         (1 - \Delta)f_1 = -i(\pa_1 - i\pa_2)g_4 - i\pa_3g_3 + g_1, \\
         (1 - \Delta)f_2 = -i(\pa_1 + i\pa_2)g_3 - i\pa_3g_4 + g_2, \\
         (1 - \Delta)f_3 = -i(\pa_1 - i\pa_2)g_2 - i\pa_3g_1 - g_3, \\
         (1 - \Delta)f_4 = -i(\pa_1 + i\pa_2)g_1 - i\pa_3g_2 - g_4, \\
    \end{cases}
\end{equation}
     where $\pa_j = \pa/\pa x_j$. Here each equation in (\ref{4eqf}) should be viewed
     as equations in $\cS'(\mbbR^3)$. As has been shown in the proof of Lemma 4.1,
     the differential operator $1 - \Delta$ has the inverse $(1 - \Delta)^{-1}$
     as a pseudodifferential operator with symbol $(1 + |\xi|^2)^{-1}$, and hence,
     by applying $(1 - \Delta)^{-1}$  and $\pa_j$ to each of the equations in (\ref{4eqf}),
     it follows that
\begin{equation}\label{4eqf2}                                              
    \begin{cases}
         \partial_jf_1 = -i(\partial_j\partial_1 - i\partial_j\partial_2)
                (1-\Delta)^{-1}g_4 - i\partial_j\partial_3(1-\Delta)^{-1}g_3
                +\partial_j(1-\Delta)^{-1}g_1,\\
         \partial_jf_2 = -i(\partial_j\partial_1 + i\partial_j\partial_2)
                  (1-\Delta)^{-1}g_3 + i\partial_j\partial_3(1-\Delta)^{-1}g_4
                +\partial_j(1-\Delta)^{-1}g_2,\\
         \partial_jf_3 = -i(\partial_j\partial_1 - i\partial_j\partial_2)
                (1-\Delta)^{-1}g_2 - i\partial_j\partial_3(1-\Delta)^{-1}g_1
                -\partial_j(1-\Delta)^{-1}g_3,\\
         \partial_jf_4 = -i(\partial_j\partial_1 + i\partial_j\partial_2)
                  (1-\Delta)^{-1}g_1 + i\partial_j\partial_3(1-\Delta)^{-1}g_2
                 -\partial_j(1-\Delta)^{-1}g_4.
    \end{cases}
\end{equation}

\SP

        (II) By Lemma 4.3 we can choose $g_0 \in L^1(\mbbR^3)\setminus h^1(\mbbR^3)$
     such that $r'_3g_0 = \pa_3(1 - \Delta)^{-1/2}g_0 \notin L^1(\mbbR^3)$. Then define
     $\mbg = {}^t(g_1, g_2, g_3, g_4) \in [L^1(\mbbR^3)]^4$ by
\begin{equation*}
        g_1(x) = g_3(x) = g_4(x) = 0 \ \ \ {\rm and} \ \ \ g_2(x) = g_0(x).
\end{equation*}
     Then we have from (\ref{4eqf2})
\begin{equation}\label{4pa3}                                                
    \pa_jf_4 = i\partial_j\partial_3(1-\Delta)^{-1}g_2 \qquad (j = 1, 2, 3).
\end{equation}
     Since $r_3'g_2 \notin L^1(\mbbR^3)$, we have necessarily $r_3'g_2 \notin h^1$. Then
     $\big(r'_1(r'_3g_2), r'_2(r'_3g_2), r'_3(r'_3g_2)\big)$ does not belong to
     $[L^1({\mbbR}^3)]^3$. It follows from (\ref{4r'jdef}) that the
     symbol $s_{j3}(\xi)$ of $r'_jr'_3$ is given by
\begin{equation*}
          s_{j3}(\xi) = \frac{i\xi_j}{(1 + |\xi|^2)^{1/2}}\frac{i\xi_3}{(1 + |\xi|^2)^{1/2}}
                               = (i\xi_j)(i\xi_3)(1 + |\xi|^2)^{-1},
\end{equation*}
     and hence by using (\ref{4r'jdef}) again, we see that
\begin{equation}\label{4r'jr'3}                                                
        r'_jr'_3g_2 = (\pa_j(1 - \Delta)^{-1/2})(\pa_3(1 - \Delta)^{-1/2})g_2
                                   = \pa_j\pa_3(1 - \Delta)^{-1}g_2
\end{equation}
     for $j = 1, 2, 3$. Thus we have from (\ref{4pa3}) and (\ref{4r'jr'3})
\begin{equation*}
      (\pa_1f_4, \pa_2f_4, \pa_3f_4) = i(r'_1r'_3g_2, r'_2r'_3g_2, r'_3r'_3g_2)
                                          \notin [L^1(\mbbR^3)]^3,
\end{equation*}
     which implies that $\mbf \notin [H^{1,1}(\mbbR^3)]^4$. This completes the
     proof of Proposition 4.4.
\end{proof}                                           

\SP

\begin{prop}                                                 
        Let $\Omega$ be an open subset of ${\mathbb R}^3$.
Then

        {\rm (i)} $[H_0^{1,1}(\Omega)]^4$  is a proper subspace of ${\mathbb H}_0^{1,1}(\Omega)$.

        {\rm (ii)} $[H^{1,1}(\Omega)]^4$  is a proper subspace of ${\mathbb H}^{1,1}(\Omega)$.
\end{prop}                                                  

\begin{proof}                                             
        (I) We are going to show the norms $\|f\|_{S,1,1,\Omega}$
     of $[H_0^{1,1}(\Omega)]^4$ and $\|f\|_{D,1,1,\Omega}$ of ${\mathbb H}_{0}^{1,1}(\Omega)$
     are not equivalent on $[C_0^{\infty}(\Omega)]^4$ (see (\ref{1SOB}) and (\ref{1S1p}) for
     the definition of these norms). To this end we use Proposition 4.4. Without loss of
     generality, we may assume that $\Omega$ contains the unit ball $\{x\,:\, |x| \leq 1\}$
     with center at the origin. As in (\ref{1pOnorm}) (with $p = 1$), we denote the norm of
     $[L^1(\Omega)]^4$ by $\|\mbf\|_{1,\Omega}$, {\it i.e.,}
\begin{equation*}
         \|\mbf\|_{1,\Omega} = \int_{\Omega} \, \sum_{j=1}^4 |f_j(x)| \, dx
                                   \qquad (\mbf(x) = {}^t(f_1(x), f_2(x), f_3(x), f_4(x)).
\end{equation*}
     By Proposition 4.4 and the fact that $[C_0^{\infty}({\mbbR}^3)]^4$ is dense in
     both $[H^{1,1}(\mbbR^3)]^4$ and $\mathbb{H}^{1,1}(\mbbR^3)$ (Theorem 1.3, (i)), the
     norms $\|{\mbf}\|_{S,1,1,\mbbR^3}$ and $\|{\mbf}\|_{D,1,1,\mbbR^3}$ are not equivalent
     on $[C_0^{\infty}({\mbbR}^3)]^4$. Therefore, by taking note of Proposition 2.2,
     (\ref{2est}) with $\Omega = \mbR^3$, which says that the norm $\|{\mbf}\|_{D,1,1,\mbbR^3}$
     is dominated by the norm $\|{\mbf}\|_{S,1,1,\mbbR^3}$,
     there exists a sequence $\{{\mbf}_n\}_{n=1}^{\infty}$ of functions in
     $[C_0^{\infty}({\mbbR}^3)]^4$ such that $\mbf_n \ne 0$ and
     $\|{\mbf_n}\|_{S,1,1,\mbbR^3} \ge (n + 1)\|{\mbf_n}\|_{D,1,1,\mbbR^3}$\,, or
\begin{equation}\label{4ineqr3}                                                      
        \|\mbf_n\|_{1,\mbbR^3} + \|\nabla\mbf_n\|_{1,\mbbR^3}
                \ge (n + 1)[\|\mbf_n\|_{1,\mbbR^3} + \|(\alpha\cdot{\pp})\mbf_n\|_{1,\mbbR^3}]
\end{equation}
     for each $n = 1, 2, \cdots$. Each $\mbf_n$ has support in some ball $\{ x \,:\, |x|\le R_n\}$
     with radius $R_n > 0$ and center at the origin. We may assume with no loss of generality that
     $R_n \ge 1$ ($n = 1, 2, \cdots$). Put $\mbg_n(x) = R_n^3\mbf_n(R_nx)$. Then $\mbg_n$ has support
     in the unit ball $\{ x \,:\, |x|\le 1 \}$, and hence in $\Omega$, so that
     $\{\mbg_n\} \subset [C_0^{\infty}(\Omega)]^4$ for each $n$. We have
$$
       \|\mbg_n\|_{1,\Omega} = \|\mbf_n\|_{1,\mbbR^3} \ \ \ {\rm and} \ \ \
          \|\partial_j \mbg_n\|_{1,\Omega}
             = R_n\|\partial_j\mbf_n\|_{1,\mbbR^3}
$$
     for $j = 1, 2, 3$. Then by (\ref{4ineqr3}) we have
$$
       \|\mbg_n\|_{1,\Omega} + \frac1{R_n}\|\nabla\mbg_n\|_{1,\Omega}
          \ge (n + 1)[ \|\mbg_n\|_{1,\Omega} +
                       \frac1{R_n}\|(\alpha\cdot \pp)\mbg_n\|_{1,\Omega}],
$$
     and hence, by noting $R_n \ge 1$
\begin{multline*}
    \hskip 50pt  \frac1{R_n}\|\nabla\mbg_n\|_{1,\Omega}
          \ge n\|\mbg_n\|_{1,\Omega}
                     + \frac{n+1}{R_n}\|(\alpha\cdot \pp)\mbg_n\|_{1,\Omega} \\
          \ge \frac{n}{R_n} [\|\mbg_n\|_{1,\Omega}
                         + \|(\alpha\cdot \pp)\mbg_n\|_{1,\Omega}\|]. \hskip 126pt
\end{multline*}
     Therefore
$$
           \|\nabla\mbg_n\|_{1,\Omega}
             \ge n[\|\mbg_n\|_{1,\Omega} + \|(\alpha\cdot \pp)\mbg_n\|_{1,\Omega}],
$$
     which implies that $\|\mbg_n\|_{S,1,1,\Omega} \ge n\|\mbg_n\|_{D,1,1,\Omega}$
     for $n = 1, 2, \cdots$. This proves that the norms $\|f\|_{S,1,1,\Omega}$ and
     $\|f\|_{D,1,1 \Omega}$ are not equivalent on $[C_0^{\infty}(\Omega)]^4$, showing (i).

        (II) As we have seen in the proof of (i), the norms of $[H_0^{1,1}(\Omega)]^4$
     and ${\mathbb H}_0^{1,1}(\Omega)$ are not equivalent. In fact, there exists a sequence
     $\{g_n\} \subset [C_0^{\infty}(\Omega)]^4$ such that
\begin{equation*}
  \|g_n\|_{S,1,1,\Omega} \geq n \|g_n\|_{D,1,1,\Omega}, \,\, n=1,2,\, \dots.
\end{equation*}
     Since clearly this sequence is also contained in $[H^{1,1}(\Omega)]^4$,
     this implies that the norms of $[H^{1,1}(\Omega)]^4$ and ${\mathbb H}^{1,1} (\Omega)$ are
     not equivalent. This shows (ii), completing the proof of Proposition 4.5.
\end{proof}                                         

\begin{proof} [Proof of Theorem 1.3, (iii)]
        Theorem 1.3, (iii) follows from Propositions 4.5.
\end{proof}

\bigskip\bigskip
        {\bf Acknowledgement.}\ 
Being deeply grieved that our respected friend, Professor Tetsuro Miyakawa,
suddenly passed away in February 2009, 
we gratefully remember nice useful discussion with him
about the Riesz transform at the early stage of the present work.
Thanks are also due to Professor Shuichi Sato for valuable
 discussion about the Hardy space and local Hardy space. 
 The research of T.I. is supported in part by JSPS Grant-in-Aid
   for Sientific Research No. 17654030.
We wish to thank the referee for his kindly pointing out
a flaw in an assertion in Theorem 1.3(ii) of our original version of 
the paper.


\end{document}